\documentclass[11pt]{amsart}

\usepackage{amsfonts,amsmath,amssymb}

\usepackage{enumerate}

\usepackage{tikz-cd}

\numberwithin{equation}{section}

\newtheorem{theorem}{Theorem}[section]
\newtheorem{lemma}[theorem]{Lemma}
\newtheorem{proposition}[theorem]{Proposition}

\newtheorem{definition}{Definition}[section]
\newtheorem{corollary}[theorem]{Corollary}
\newtheorem{remark}[theorem]{Remark}

\newcommand{\cl}[1]{\mathcal{#1}} %\cl A
\newcommand{\bb}[1]{\mathbb{#1}}

\newcommand{\sca}[1]{\left\langle#1\right\rangle} %\sca{x,y}
\newcommand{\nor}[1]{\left\Vert #1\right\Vert}

\begin{document}

\title{HILBERT MODULES, RIGGED MODULES AND STABLE ISOMORPHISM}

\author[G. K. Eleftherakis and E. Papapetros ]{G. K. Eleftherakis and E. Papapetros }

\address{G. K. Eleftherakis\\ University of Patras\\Faculty of Sciences\\ Department of Mathematics\\265 00 Patras, Greece }
\email{gelefth@math.upatras.gr} 

\address{E. Papapetros\\ University of Patras\\Faculty of Sciences\\ Department of Mathematics\\265 00 Patras, Greece }

\email{e.papapetros@upatras.gr} 

\thanks{2020 {\it Mathematics Subject Classification.} 47L30 (primary), 46L05, 46L08, 47L05, 47L25,  
16D90 (secondary)} 

\keywords{Operator algebras, $C^*$-algebras, Hilbert modules, TRO, Stable isomorphism, Morita equivalence}

\begin{abstract}  
Rigged modules over an operator algebra are a generalization of Hilbert modules over a $C^{\star}$-algebra. We characterize the rigged modules over an operator algebra $\cl A$ which are orthogonally complemented in $C_\infty(\cl A),$  the space of infinite columns with entries in $\cl A.$ We show that every such rigged module `restricts' to a bimodule of Morita equivalence between appropriate stably isomorphic operator algebras. 
\end{abstract}

\date{}

\maketitle

\section{Introduction}

Let $X, Y$ be operator spaces. We call them stably isomorphic if the spatial tensor products $X\otimes
\cl K, Y\otimes\cl K$  are completely isometrically isomorphic, where $\cl K$ is the algebra of compact operators acting on an infinite dimensional Hilbert space. We also denote by $C_\infty(X)$ the operator space of infinite columns with entries in $X.$ In case $X$  is a right rigged module over an operator algebra $\cl A,$ so is $C_\infty(X).$

The notion of a Hilbert $C^{\star}$-module was developed in the early 1970s by Paschke and Rieffel, see \cite{Rie82,Paschke}. A Hilbert module over a $C^*$-algebra $\cl A$ is a right $\cl A$-module $Y$ together with a map $\sca{\cdot ,\cdot }:  Y\times  Y\rightarrow \cl A$ which is linear in the second variable, and which also satisfies the following conditions:

(1) $\sca{y,y}\geq 0$ for all $y\in  Y,$

(2) $\sca{y,y}= 0\Leftrightarrow y=0,$

(3) $\sca{y,za}=\sca{y,z}a, $ for all $y, z \in  Y, a \in \cl A,$

(4) $\sca{y,z}^*=\sca{z,y}$ for all $y, z \in  Y,$

(5) $ Y$ is complete in the norm $\|y\|=\|\sca{y,y}\|^{\frac{1}{2}}.$

Observe that the space $I_{\cl A}( Y)$, which is the closure of the linear span of the set $\{\sca{y,z},  y, z \in  Y\}$, is an ideal of $\cl A.$  

Consider the  $C^*$-algebra $\bb K_{\cl A}( Y)$ of the `compact' adjointable operators from $ Y$ 
to $ Y.$ It is known that $ Y$ is a bimodule of Morita equivalence  between $I_{\cl A}( Y)$ and $\bb K_{\cl A}( Y).$  But these $C^*$-algebras are not always stably isomorphic.

Let $Y$ be a right Hilbert $\cl A$-module. In case there exists a sequence $(y_k)\subseteq Y$ 
such that 
$$y=\sum_ky_k\sca{y_k,y}, \,\,\forall\,y\,\in \,Y$$
where the series converges in the norm of $Y$, we say that $(y_k)$ is a right quasibasis for $Y.$  It follows by the Brown--Kasparov stabilization theorem, see Corollary 8.20 in \cite{BleLeM04}, that the spaces $I_{\cl A}( Y)$, 
$\bb K_{\cl A}( Y), Y$ are all stably isomorphic. 

Let $Y$ be a right Hilbert $\cl A$-module. We call it countably generated if there exists a sequence $(y_k)\subseteq Y$ such that $$Y=\overline{span}\{y_ka: k\in \bb N, a\in \cl A\}.$$
 If $Y$ has a right quasibasis, then $Y$ is countably generated, and conversely. Every countably generated Hilbert  $\cl A$-module is isomorphic as a Hilbert $\cl A$-module with an orthogonally complemented bimodule of $C_\infty(\cl A).$
 
 Blecher in \cite{Ble-Gen} generalized the notion of Hilbert modules to the setting of non-selfadjoint operator algebras. He called these modules rigged modules, see the definition below. Hilbert modules over a $C^*$-algebra are rigged modules in terms of this definition. Using the notion of a ternary ring of operators, we introduce a new category of $\cl A-$rigged modules, where $\cl A$ is an operator algebra, the $\sigma\Delta-\cl A-$rigged modules. We prove that an  $\cl A-$rigged module is a $\sigma\Delta-\cl A-$rigged module if and only if it is isomorphic with an  orthogonally complemented module in $C_\infty(\cl A).$ We also introduce a subcategory of the $\sigma\Delta-\cl A-$rigged modules, the doubly $\sigma\Delta-\cl A-$ rigged modules. In the case of $C^*$-algebras these two categories coincide. 
 Every doubly $\sigma\Delta-\cl A-$rigged module implements a stable isomorphism between the corresponding operator algebras. Conversely, if $\cl A$ and $\cl B$ are stably isomorphic operator algebras, there exists a doubly $\sigma\Delta-\cl A-$rigged module $Y$ which is a bimodule of strong Morita equivalence (BMP-Morita equivalence) for $\cl A$ and $\cl B$  in the sense of Blecher, Muhly and Paulsen, \cite{BMP00}.
 Every $\sigma\Delta-\cl A-$rigged module has a `restriction' which is a doubly $\sigma\Delta-\cl A-$rigged module. Thus every orthogonally complemented rigged module  in $C_\infty(\cl A),$ has a  `restriction' making it into a bimodule of BMP-Morita equivalence over some operator algebras $\cl C,\cl D.$ Furthermore, $\cl C$ and $\cl D$ are stably isomorphic. 
 
In Section \ref{444}, we will develop a theory of Morita equivalence for rigged modules.
If $\cl A\,,\cl B$ are operator algebras,  $\,E$ is a right $\cl B$-rigged module, and $F$ is a right $\cl A$-rigged module, we call $E$ and $F$ $\sigma$-Morita equivalent if there exists a doubly $\sigma \Delta$-$\cl A$-rigged module $ Y$ such that $\cl A\cong \tilde{ Y}\otimes_{\cl B}^h  Y\,,\cl B\cong  Y\otimes_{\cl A}^h \tilde{ Y}$ and also $F\cong E\otimes_{\cl B}^h  Y,$  where $\tilde{ Y}$ is the counterpart bimodule of $Y.$ In this case we write $E\sim_{\sigma M} F.$ We will prove that if $E\sim_{\sigma M} F$, then $E$ and $F$ are stably isomorphic. 

This paper has been written with an emphasis on the theory of non-selfadjoint operator algebras, but the conclusions for $C^*$-algebras follow easily. 

At this point, we recall some definitions, notation and lemmas which will be useful for what follows. 

\begin{definition}, \cite{Ble-Gen}
\label{rigged}

Let $\cl A$ be an approximately unital operator algebra, i.e. an operator algebra with a contractive approximate identity,  and let $Y$ be a right $\cl A$-operator module.
Suppose there is a net $(n(b))_{b\in B}$ of positive integers and right $\cl A$-module maps $$\Phi_{b}:Y\to C_{n(b)}(\cl A)\,\,,\Psi_{b}:C_{n(b)}(\cl A)\to Y\,,b\in B$$ such that:\\
i) the maps $\Phi_{b}\,,\Psi_{b}$ are completely contractive;\\
ii) $\,\Psi_{b}\circ \Phi_{b}\to Id_{Y}$ strongly on $Y$;\\
iii) the maps $\Psi_{b}\,,b\in B$ are right $\cl A$-essential maps (that is, $\Psi_{b}\,e_i\to \Psi_{b}$ for a bounded approximate identity $(e_i)_{i\in I}$ of $\cl A$);\\
iv)$\,\Phi_{c}\circ \Psi_{b}\circ \Phi_{b}\to \Phi_{c}\,,\forall\,c\in B$ (uniformly in norm).\\
Then we say that $Y$ is a right $\cl A$-rigged module.

\end{definition}

 We denote by $\mathbb{B}(H,K)$ the space of all linear and bounded operators from the Hilbert space $H$ to the Hilbert space $K$.
If $H=K$, we write $\mathbb{B}(H,H)=\mathbb{B}(H)$.
If $ X$ is a subset of $\mathbb{B}(H,K)$ and $ Y$ is a subset of $\mathbb{B}(K,L)$, then we denote by $\overline{[ Y\, X]}$ the norm-closure of the linear span of the set $$\left\{y\,x\in\mathbb{B}(H,L)\,,y\in Y\,,x\in X\right\}.$$ Similarly, if $ Z$ is a subset of $\mathbb{B}(L,R)$, we define the space $\overline{[ Z\,Y\, X]}$.

\begin{definition}

$i)\,$ A linear subspace $M\subseteq \mathbb{B}(H,K)$ is called a ternary ring of operators (TRO) if $M\,M^{\star}\,M\subseteq M$.\\
$ii)\,$ A norm closed ternary ring of operators $M$ is called a $\sigma$-TRO if there exist sequences $\left\{m_i\in M\,,i\in\mathbb{N}\right\}$ and $\left\{n_j\in M\,,j\in\mathbb{N}\right\}$ such that $$\lim_{n}\sum_{i=1}^n m_i\,m_i^{\star}\,m=m\,\,,\lim_{t}\sum_{j=1}^t m\,n_j^{\star}\,n_j=m\,,\forall\,m\in M$$ and $$\nor{\sum_{i=1}^n m_i\,m_i^{\star}}\leq 1\,,\nor{\sum_{j=1}^{t}n_j^{\star}\,n_j}\leq 1\,,\forall\,n\,,t\in\mathbb{N}.$$

\end{definition}
A norm closed TRO $M$ is a $\sigma$-TRO if and only if the $C^{\star}$-algebras $\overline{[M^{\star}\,M]}$ and $\overline{[M\,M^{\star}]}$ have $\sigma$-units, \cite{Bro77}.

If $ X$ is an operator space, then the spatial tensor product $ X\otimes \cl K$ is completely isometrically isomorphic with the space $ K_{\infty}( X)$, which is the norm closure of the finitely supported matrices in $\mathbb{M}_{\infty}( X)$.
Here, $\mathbb{M}_{\infty}( X)$ is the space of $\infty\times \infty$ matrices with entries in $X$ which define bounded operators.
Also, for $ Y$ another operator space, we denote by $ X\otimes^h  Y$ the Haagerup tensor product of $ X$ and $ Y$.
If $\cl A$ is an operator algebra, $ X$ is a right $\cl A$-module, and $ Y$ is a left $\cl A$-module, then we denote by $ X\otimes_{\cl A}^h  Y$ the balanced Haagerup tensor product of $ X$ and $ Y$ over $\cl A$, see \cite{BMP00}.
We now give two basic definitions.

\begin{definition}

Let $ X\subseteq \mathbb{B}(H,K)\,, Y\subseteq \mathbb{B}(L,R)$ be operator spaces.
We call them $\sigma$-TRO equivalent if there exist $\sigma$-TROs $M_1\subseteq \mathbb{B}(H,L)\,,M_2\subseteq \mathbb{B}(K,R)$ such that $$ X=\overline{[M_2^{\star}\, Y\,M_1]}\,\,, Y=\overline{[M_2\, X\,M_1^{\star}]}.$$ In this case we write $ X\sim_{\sigma TRO} Y$.

\end{definition}

\begin{definition}
\label{eq}

Let $ X, Y$ be operator spaces.
We call them $\sigma \Delta$ equivalent if there exist completely isometric maps $\phi: X\to \mathbb{B}(H,K)\,,\psi: Y\to \mathbb{B}(L,R)$ such that $\phi(X)\sim_{\sigma TRO}\psi( Y)$, and we then write $ X\sim_{\sigma \Delta} Y$.

\end{definition}

If $\cl A\,,\cl B$ are abstract or concrete operator algebras, we say that they are $\sigma \Delta$ equivalent and we write $\cl A\sim_{\sigma \Delta}\cl B$ if there exist completely isometric representations $a:\cl A\to a(\cl A)\subseteq \mathbb{B}(H)\,,\beta:\cl B\to \beta(\cl B)\subseteq \mathbb{B}(K)$ and a $\sigma$-TRO $M\subseteq \mathbb{B}(H,K)$ such that $$a(\cl A)=\overline{[M^{\star}\,\beta(\cl B)\,M]}\,\,,\beta(\cl B)=\overline{[M\,a(\cl A)\,M^{\star}]}.$$

For further details about the notion of $\sigma \Delta$ equivalence of operator algebras and operator spaces, we refer the reader to \cite{Ele14, Elest, Elekak, Ele-Pap}.
If $ X\,, Y$ are operator spaces, then $ X\sim_{\sigma \Delta} Y$ if and only if $ X$ and $ Y$ are stably isomorphic, that is, $ K_{\infty}( X)\cong  K_{\infty}( Y)$ (similarly for operator algebras).
We now present a lemma  which will be used in some of the proofs in the following sections.

\begin{lemma}
\label{a lot}

Suppose that $\cl A\,,\cl B$ are operator algebras and $D\subseteq \cl B$ is a $C^{\star}$-algebra such that $\overline{[D\,\cl B]}=\overline{[\cl B\,D]}=\cl B$.
Let $M\subseteq \mathbb{B}(H,K)$ be a $\sigma$-TRO such that $\overline{[M^{\star}\,M]}\cong D$ (as $C^{\star}$-algebras) and assume that $\cl A\cong M\otimes_{D}^h \cl B\otimes_{ D}^h M^{\star}$.
Then, $\cl A\sim_{\sigma \Delta}\cl B.$

\end{lemma}

A proof of this lemma can be found in \cite[Lemma 2.2]{Ele-Pap}.

\section{Orthogonally complemented modules and $\sigma\Delta$ rigged  Modules}\label{222}

Let $\cl A$ be an approximately unital operator algebra and $P:C_{\infty}(\cl A)\to C_{\infty}(\cl A)$ be a left multiplier of $C_{\infty}(\cl A)$ (that is, $P\in M_{\ell}(C_{\infty}(\cl A)))$ such that $P$ is contractive and $P^2=P$.
Then the space $W=P(C_{\infty}(\cl A))$ is said to be orthogonally complemented in $C_{\infty}(\cl A).$
In this section we characterize the orthogonally complemented modules in the the terms of ternary rings of operators. A dual version of the results obtained here is in \cite{Ble-Kr}. 

\begin{definition}
\label{basic-def}

Let $\cl A\subseteq \mathbb{B}(H)$ be an approximately unital operator algebra and $M\subseteq \mathbb{B}(H,K)$ be a $\sigma$-TRO such that $M^{\star}\,M\,\cl A\subseteq \cl A.$ The operator space $Y_0=\overline{[M\,\cl A]}\subseteq \mathbb{B}(H,K)$ is called a $\,\sigma$-TRO-$\cl A$-rigged module.

\end{definition}

We recall that $Y_0$ is a right $\cl A$-operator module with action $$(m\,a)\cdot x=m(a\,x)\,,m\in M\,,a\,,x\in\cl A.$$

\begin{definition}
\label{abstract-rigged}
Let $\cl A$ be an abstract  approximately unital operator algebra and let $Y$ be an abstract right $\cl A$-module.
We call $\,Y$ a $\,\sigma \Delta$-$\cl A$-rigged module if there exists a completely isometric homomorphism $a:\cl A\to a(\cl A)$ and there exist a $\,\sigma$-TRO-$a(\cl A)$-rigged module $Y_0$ and a complete surjective isometry $\rho:Y\to Y_0$ which is also a right $\cl A$-module map. In case $\cl A$ 
is a $C^*-$algebra we call $Y$ a $\,\sigma \Delta$-$\cl A$-Hilbert module.
\end{definition}

\begin{proposition}
\label{basic-rigged}

Let $\cl A$ be an approximately unital operator algebra.
Every $\sigma \Delta$-$\cl A$-rigged module is a right rigged module over $\cl A$ in the sense of Definition \ref{rigged}.
 
\end{proposition}

\begin{proof}
Let $Y$ be a right $\sigma \Delta$-$\cl A$-rigged module.
Then there exist a completely isometric homomorphism $a:\cl A\to a(\cl A)\subseteq \mathbb{B}(H)$, a $\,\sigma$-TRO $M\subseteq \mathbb{B}(H,K)$, and a complete surjective isometry $\rho:Y\to Y_0=\overline{[M\,a(\cl A)]}$ which is also a right $\cl A$-module map.
So, if we choose a $\left\{\Phi_{b}\,,\Psi_{b}: b\in B\right\}$ for the module $Y_0$, then we define for each $b\in B$ a map $\Phi_{b}'=\Phi_{b}\circ \rho\,\,,\Psi_{b}'=\rho^{-1}\circ \Psi_{b}$ and we can see that the $\left\{\Phi_{b}'\,,\Psi_{b}': b\in B\right\}$ satisfy the conditions of Definition \ref{rigged}.
So, $Y$ becomes a right $\cl A$-rigged module.
Therefore it suffices to prove the proposition when $Y=\overline{[M\,a(\cl A)]}\subseteq \mathbb{B}(H,K)$.
Since $M$ is a $\sigma$-TRO, there exists a sequence $\left\{m_i\in M: i\in\mathbb{N}\right\}$ such that $\,||(m_i)_{i\in\mathbb{N}}||\leq 1$ and $$\sum_{i=1}^{\infty}m_i\,m_i^{\star}\,m=m\,,\forall\,m\in M.$$ Since $Y=\overline{[M\,a(\cl A)]}$, it follows that $$\sum_{i=1}^{\infty}m_i\,m_i^{\star}\,y=y\,,\forall\,y\in Y.$$ For $n\in\mathbb{N}$ we define $$\Phi_{n}:Y\to C_{n}(\cl A)\,,\Phi_{n}(y)=\begin{pmatrix}m_{1}^{\star}\,y\\
...\\
m_{n}^{\star}\,y\end{pmatrix},$$ which is linear and a completely contractive right $\cl A$-module map.
We also define the linear, completely contractive and right $\cl A$-module map $$\Psi_{n}:C_{n}(\cl A)\to Y\,,\Psi_{n}\left(\begin{pmatrix}a_1\\
...\\
a_{n}\end{pmatrix}\right)=\sum_{i=1}^{n}m_i\,a_i.$$

For all $y\in Y$, it holds that $$\Psi_{n}\circ \Phi_{n}(y)=\Psi_{n}\left(\begin{pmatrix}m_1^{\star}\,y\\
...\\
m_{n}^{\star}\,y\end{pmatrix}\right)=\sum_{i=1}^{n}m_i\,m_i^{\star}\,y\to y=Id_{Y}(y)$$ and we conclude that $\Psi_{n}\circ \Phi_{n}\to Id_{Y}$ strongly on $Y.$ The next step is to prove that $\Psi_{n}\,,n\in\mathbb{N}$, is a right $\cl A$-essential map.
To this end, let $(e_i)_{i\in I}$ be a contractive  approximate identity of $\cl A.$ We have that
\begin{align*}
    \nor{\Psi_{n}e_{i}\left(\begin{pmatrix}a_1\\
    ...\\
    a_{n}\end{pmatrix}\right)-\Psi_{n}\left(\begin{pmatrix} a_1\\
    ...\\
    a_{n}\end{pmatrix}\right)}&=\nor{\Psi_{n}\left(\begin{pmatrix}a_1\\
    ...\\
    a_{n}\end{pmatrix}\right)\,e_i-\Psi_{n}\left(\begin{pmatrix}a_1\\
    ...\\
   a_{n}\end{pmatrix}\right)}\\&=\nor{\sum_{j=1}^{n}(m_j\,a_j)\,e_i-\sum_{j=1}^{n}m_j\,a_j}\\&=\nor{\sum_{j=1}^{n}m_j\left(a_j\,e_i-a_j\right)}\\&\leq \sum_{j=1}^{n}||m_j||\,||a_j\,e_i-a_j||
\end{align*}
where $$\lim_{i}||a_j\,e_i-a_j||=0$$ for all $j=1,...,n$, so $$\lim_{i}\nor{\Psi_{n}e_{i}\left(\begin{pmatrix}a_1\\
    ...\\
    a_{n}\end{pmatrix}\right)-\Psi_{n}\left(\begin{pmatrix} a_1\\
    ...\\
    a_{n}\end{pmatrix}\right)}=0.$$
Finally, let $r\in\mathbb{N}$.
We shall show that $$\lim_{n}||\Phi_{r}\circ \Psi_{n}\circ \Phi_{n}-\Phi_{r}||=0.$$ 
We denote by $s_n$ the operators $$s_n=\sum_{i=1}^n m_i\,m_i^{\star}\,,n\in\mathbb{N}.$$

Hence, if $y\in Y$, we have that
\begin{align*}
    ||\Phi_{r}\circ \Psi_{n}\circ \Phi_{n}(y)-\Phi_{r}(y)||&=||\Phi_{r}\left(\Psi_{n}\circ \Phi_{n}(y)-y\right)||\\&=\nor{\begin{pmatrix}m_1^{\star}\,s_n-m_1^{\star}\\
    ...\\
    m_{r}^{\star}\,s_n-m_{r}^{\star}\end{pmatrix}\,y}\\&\leq \nor{\begin{pmatrix}m_1^{\star}\,s_n-m_1^{\star}\\
    ...\\
    m_{r}^{\star}\,s_n-m_{r}^{\star}\end{pmatrix}}\,||y||
\end{align*}
Therefore, $||\Phi_{r}\circ \Psi_{n}\circ \Phi_{n}-\Phi_{r}||\leq \nor{\begin{pmatrix}m_1^{\star}\,s_n-m_1^{\star}\\
    ...\\
    m_{r}^{\star}\,s_n-m_r^{\star}\end{pmatrix}}$

 Since    
 $$\lim_{n}||m_i^{\star}\,s_n-m_i^{\star}||=0\,,\forall\,i=1,...,r,$$ we have that 
    $$\lim_{n}||\Phi_{r}\circ \Psi_{n}\circ \Phi_{n}-\Phi_{r}||=0.$$
We conclude that $Y$ is a right $\cl A$-rigged module in the sense of Definition \ref{rigged}.
\end{proof}

\begin{theorem}
\label{orthog}

Let $\cl A$ be an approximately unital operator algebra and $Y$ be a right $\cl A$-operator module.
Then the following are equivalent:\\
$i)$\ $Y$ is a right $\sigma \Delta$-$\cl A$-rigged module.\\
$ii)$\ $Y$ is orthogonally complemented in $C_{\infty}(\cl A)$.

\end{theorem}

\begin{proof}

$i)\implies ii)$

 Let $a:\cl A\to a(\cl A)\subseteq \mathbb{B}(H)$ be a completely isometric representation of $A$ on $H$ and assume there is a $\sigma$-TRO $M\subseteq \mathbb{B}(H,K)$ such that $M^{\star}\,M\,a(\cl A)\subseteq a(\cl A).$ Consider the $\sigma \Delta$-$\cl A$-rigged module $Y_0=\overline{[M\,a(\cl A)]}\subseteq \mathbb{B}(H,K)$ and a complete surjective isometry $\phi:Y\to Y_0$ which is also a right $\cl A$-module map.
Let $\left\{m_i\in M: i\in\mathbb{N}\right\}$ be a sequence of elements of $M$ having the property $$\nor{\sum_{i=1}^n m_i\,m_i^{\star}}\leq 1\,,\forall\,n\in\mathbb{N}\,\,,\sum_{i=1}^{\infty}m_i\,m_i^{\star}\,m=m\,,\forall\,m\in M.$$
It follows that $$\sum_{i=1}^{\infty}m_i\,m_i^{\star}\,y=y\,,\forall\,y\in Y_0.$$

We define the map $f:Y_0\to C_{\infty}(\alpha(\cl A))$ by $f(y)=(m_i^{\star}\,y)_{i\in\mathbb{N}}$, which is linear and a $\cl A$-module map.
Also, $$||f(y)||^2=\nor{\sum_{i=1}^{\infty}(m_i^{\star}\,y)^{\star}\,m_i^{\star}\,y}=\nor{\sum_{i=1}^{\infty}y^{\star}\,m_i\,m_i^{\star}\,y}=||y^{\star}\,y||=||y||^2,$$
so $f$ is an isometry.
We also define $$g:C_{\infty}(\alpha(\cl A))\to Y_0\,,g((\alpha(x_i))_{i\in\mathbb{N}})=\sum_{i=1}^{\infty}m_i \alpha(x_i),$$ which is linear and a contractive $\cl A$ right module map.
We see that $$(g\circ f)(y)=g((m_i^{\star}\,y)_{i\in\mathbb{N}})=\sum_{i=1}^{\infty}m_i\,m_i^{\star}\,y=y\,,\forall\,y\in Y_0,$$ that is, $g\circ f=Id_{Y_0}$.
We now define $P=f\circ g:C_{\infty}(\alpha(\cl A))\to C_{\infty}(\alpha(\cl A)).$ Clearly $P$  is a  contractive map satisfying $P^2=P$.
We shall prove that $P\in M_{\ell}(C_{\infty}(\cl A)).$ For all $x=\begin{pmatrix}x_1\\
x_2\\
...\end{pmatrix}\in C_{\infty}(a(\cl A))$ we have that $$P(x)=\left(m_i^{\star}\,\sum_{j=1}^{\infty}m_j\,x_j\right)_{i\in\mathbb{N}}=s\,x,$$ where $s=\left(m_i^{\star}\,m_j\right)_{i\,,j=1}^{\infty}\in \mathbb{M}_{\infty}(\mathbb{B}(H))$.
Observe that $s=\begin{pmatrix}m_1^{\star}\\
m_2^{\star}\\
...\end{pmatrix}\left(m_1,m_2,...\right)$ and due to the fact that $||(m_1,m_2,...)||\leq 1$ we get $||s||\leq 1.$ We define the map $$\tau_{P}:C_2(C_{\infty}(a(\cl A)))\to C_2(C_{\infty}(a(\cl A)))\,,\tau_{P}\left(\begin{pmatrix}x\\
y\end{pmatrix}\right)=\begin{pmatrix}P(x)\\
y\end{pmatrix}=\begin{pmatrix}s\,x\\
y\end{pmatrix}$$ and for all $\begin{pmatrix}x\\
y\end{pmatrix}\in C_{2}(C_{\infty}(a(\cl A)))$ holds that $$\nor{\tau_{P}\left(\begin{pmatrix}x\\
y\end{pmatrix}\right)}=\nor{\begin{pmatrix}s\,x\\
y\end{pmatrix}}=\nor{\begin{pmatrix} s & 0\\
0 & I_{2}\end{pmatrix}\,\begin{pmatrix}x\\
y\end{pmatrix}}\leq \nor{\begin{pmatrix}x\\
y\end{pmatrix}},$$
so $\tau_{P}$ is a contraction.
Similarly, we can prove that $\tau_{P}$ is completely contractive.
Therefore by \cite[Theorem 4.5.2]{BleLeM04} $P$ is a left multiplier of $C_{\infty}(a(\cl A)).$
It is easy to see that $f(Y_0)=P(C_{\infty}(\alpha(\cl A)))$ and thus $Y\simeq P(C_{\infty}(\alpha(\cl A))).$\\

$ii)\implies i)$ 

Suppose that $\cl A\subseteq \cl A^{\star \star}\subseteq \mathbb{B}(H).$ Let $P:C_{\infty}(\cl A)\to C_{\infty}(\cl A)$ be a left multiplier of $C_{\infty}(\cl A)$ which is a right $\cl A$-module map with $||P||_{cb}\leq 1$ and such that $P^2=P\,\,,Y\cong P(C_{\infty}(\cl A)).$ According to Appendix B of \cite{Ble-03}, there is  an extension $\tilde{P}:C_{\infty}^{w}(\cl A^{\star \star})\to C_{\infty}^{w}(\cl A^{\star \star})$ of $P.$ The operator $\tilde{P}$ lies in the diagonal of $M_{l}(C_{\infty}^{w}(\cl A^{\star \star}))$, which is contained in $ \mathbb{M}_{\infty}(\cl A^{\star \star}).$ Therefore $\tilde{P}=(\tilde{p_{i,j}})_{i,j\in\mathbb{N}}$ where $\tilde{p_{i,j}}\in \cl A^{\star \star}\,,\forall\,i\,,j\in\mathbb{N}$.
Thus, $$\tilde{P}(u)=(\tilde{p_{i,j}})_{i,j\in\mathbb{N}}\cdot u\,\,,\forall\,u=\begin{pmatrix}u_1\\ 
u_2\\
...\end{pmatrix}\in C_{\infty}(\cl A).$$ In what follows we identify $\tilde{P}$ 
and $(\tilde{p_{i,j}})_{i,j}.$ We have that $Y\cong P(C_{\infty}(\cl A))=\tilde{P}(C_{\infty}(\cl A))$ and $\tilde{P}^2=\tilde{P}=\tilde{P}^{\star}.$  Let $N_2=[\tilde{P}], D$ be the $C^{\star}$-algebra generated by $\tilde{P}$ and $\cl{K}_{\infty}$ and  let $N_1=C_{\infty}$.
By \cite[Lemma 2.5]{Elest}, $M=\overline{[N_2\,D\,N_1]}=\overline{[\tilde{P}\,D\,C_{\infty}]}$ is a $\sigma$-TRO.
We claim that $D\,C_{\infty}(\cl A)\subseteq C_{\infty}(\cl A).$ Indeed, $$\tilde{P}(C_{\infty}(\cl A))=P(C_{\infty}(\cl A))\subseteq C_{\infty}(\cl A)\,(1)$$ and $\,C_{\infty}\,R_{\infty}\,C_{\infty}(\cl A)\subseteq C_{\infty}(\cl A).$ Due to the fact that $\cl{K}_{\infty}=C_{\infty}\,R_{\infty}$, we have that $$\cl{K}_{\infty}\,C_{\infty}(\cl A)\subseteq C_{\infty}(\cl A)\,(2).$$ But since $D$ is generated by $\tilde{P}\,,\cl{K}_{\infty}$ by $(1)\,,(2)$ we have that $D\,C_{\infty}(\cl A)\subseteq C_{\infty}(\cl A).$ Now, $$P(C_{\infty}(\cl A))\subseteq \tilde{P}\,D\,C_{\infty}(\cl A)=\overline{[M\,\cl A]}.$$ On the other hand,
\begin{align*}
    \overline{[M\,\cl A]}&=\overline{[\tilde{P}\,D\,C_{\infty}\cdot \cl A]}\\&\subseteq \tilde{P}\,(C_{\infty}(\cl A))\\&=P(C_{\infty}(\cl A))
\end{align*}
so, $$\overline{[M\,\cl A]}=P(C_{\infty}(\cl A)).$$

 Finally,
\begin{align*}
    M^{\star}\,M\,\cl A&\subseteq M^{\star}\,P(C_{\infty}(\cl A))\\&=R_{\infty}\,D\,\tilde{P}(C_{\infty}(\cl A))\\&\subseteq R_{\infty}\,D\,C_{\infty}(\cl A)\\&\subseteq R_{\infty}\,C_{\infty}(\cl A)\\&=R_{\infty}\,C_{\infty}\cdot \cl A\\&=\cl A
\end{align*}
so $Y$ is a right $\sigma \Delta$-$\cl A$-rigged module.

\end{proof}

There is a category of rigged modules, the so-called countably column generated and approximately projective modules. We are going to examine whether there is a connection between them and the $\sigma \Delta$-rigged modules.

\begin{definition}, \cite{Ble-Gen}.

Let $\cl A$ be an approximately unital operator algebra.
A right $\cl A$ operator module $Y$ is called countably column generated and approximately projective  (CCGP for short) if there are completely contractive right $\cl A$-module maps $\phi:Y\to C_{\infty}(\cl A)$ and $\psi:C_{\infty}(\cl A)\to Y$ with $\psi$ finitely $\cl A$-essential (that is, for all $n\in\mathbb{N}$ the restriction map of $\psi$ to $C_{n}(A)\subseteq C_{\infty}(\cl A)$ is right $\cl A$-essential) and also $\psi\circ \phi=Id_{Y}.$

\end{definition}

\begin{remark}

From \cite[Theorem 8.3]{Ble-Gen} and Theorem \ref{orthog}, it is obvious that a CCGP module is a $\sigma \Delta$-rigged module.
The converse is not true.
Indeed, by Theorem 8.2 of \cite{Ble-Gen}, we have that the CCGP modules over $C^{\star}$-algebras are precisely the countably generated right Hilbert modules, but there exist $\sigma \Delta$-Hilbert modules which are not countably generated. For example if $\cl A$  a $C^*-$algebra without $\sigma-$unit, since $\bb C\cl A=\cl A$ then $\cl A$ is a  $\sigma \Delta$-Hilbert module over itself, but clearly is not countably generated, so it is not a CCGP module.

\end{remark}

\section{Doubly $\sigma\Delta-$ rigged  Modules}\label{222}

In this section we introduce a subcategory of $\sigma\Delta-$ rigged  modules, the doubly $\sigma\Delta-$ rigged modules and we prove that these modules implement stable isomorphism between the corresponding operator algebras.  

\begin{definition}
\label{B-M-P-d}

Let $Y$ be a right $\cl A$-operator module over the approximately unital operator algebra $\cl A.$  We call $Y$ a BMP equivalence bimodule if there exist an operator algebra $\cl B$ such that $Y$ is a left $\cl B$-operator module and a $\cl B-\cl A$ operator module $X$ such that $$\cl B\cong Y\otimes_{\cl A}^h X\,\,,\cl A\cong X\otimes_{\cl B}^h Y.$$ In this case we call $X$ and $Y$ bimodules of BMP-Morita equivalence.

\end{definition}

We note that every $\cl B$-$\cl A$-bimodule of Morita equivalence is a right $\cl A$-rigged module.
We now introduce the notion of a doubly $\sigma \Delta$-rigged module.

\begin{definition}

\label{doubly}

Let $\cl A\subseteq \mathbb{B}(H)$ be an approximately unital operator algebra and $M\subseteq \mathbb{B}(H,K)$ be a $\sigma$-TRO such that $$M^{\star}\,M\,\cl A\subseteq \cl A\,\,,\overline{[M^{\star}\,M\,\cl A]}=\overline{[\cl A\,M^{\star}\,M]}.$$
We call the operator space $Y=\overline{[M\,\cl A]}\subseteq \mathbb{B}(H,K)$ a doubly $\sigma$-TRO-$\cl A$-rigged module.

\end{definition}

We note that every doubly $\sigma$-TRO-$\cl A$-rigged module is also a $\sigma$-TRO-$\cl A$-rigged module in the sense of Definition \ref{basic-def}.

\begin{definition}
\label{abstract-doubly}

Let $\cl A$ be an abstract approximately unital operator algebra and $Y$ be an abstract right $\cl A$-module.
We call $Y$ a doubly $\sigma \Delta$-$\cl A$-rigged module if there exists a completely isometric homomorphism $a:\cl A\to a(\cl A)$ and also there exists a doubly $\sigma$-TRO-$a(\cl A)$-rigged module $Y_0$ and  a complete onto isometry $\phi:Y\to Y_0$ which is a right $\cl A$-module map.

\end{definition}
 
\begin{definition}
\label{restriction}

Let $\cl A$ be an approximately unital operator algebra and $Y$ be a $\sigma \Delta$-$\cl A$-rigged module.
There exist $a:\cl A\to a(\cl A)\subseteq \mathbb{B}(H)$, a completely isometric representation of $\cl A$ on $H$, and a $\sigma$-TRO $M\subseteq \mathbb{B}(H,K)$ such that $M^{\star}\,M\,a(\cl A)\subseteq a(\cl A)$ and $Y\cong Y_0=\overline{[M\,a(\cl A)]}.$ Then the operator space $Z=\overline{[Y_0\,M^{\star}\,M]}\subseteq \mathbb{B}(H,K)$ is called the restriction of $\,Y$ over $\cl A.$ Observe that $Z$ is a right  module over the operator algebra $\overline{[a(\cl A)\,M^{\star}\,M]}.$

\end{definition}

In the following theorem we prove that the notions of $\sigma\Delta$-right Hilbert modules and of doubly $\sigma\Delta$-right Hilbert modules coincide:

\begin{theorem}
\label{main0}

Let $\cl A$ be a $C^{\star}$-algebra and let $Y$ be a right Hilbert module over $\cl A.$ The following are equivalent:\\
i) $Y$ is orthogonally complemented in $C_{\infty}(\cl A)$;\\
ii) $Y$ is a $\sigma \Delta$ right Hilbert module over $\cl A$;\\
iii) $Y$ is a doubly $\sigma \Delta$ right Hilbert module over $\cl A.$

\end{theorem}

\begin{proof}

$i) \implies iii)$

Let $P:C_{\infty}(\cl A)\to C_{\infty}(\cl A)$ be an adjointable map such that $P=P^2=P^{\star}$ and $Y\cong P(C_{\infty}(\cl A)).$ Since $P\in M_{l}(C_{\infty}(\cl A)),$ where $M_{l}(C_{\infty}(\cl A))$ is 
  the left multiplier algebra of $C_{\infty}(\cl A), P$ extends to a multiplier of $C_{\infty}^w(\cl A^{\star \star}).$ Here  $\cl A^{\star \star}$ is the second dual of $\cl A$  and $C_{\infty}^w(\cl A^{\star \star})$ is the space of columns with entries in  $\cl A^{\star \star}$ which define bounded operators.
The algebra of left 
 multipliers of $C_{\infty}^w(\cl A^{\star \star}) $ is isomorphic to $\bb M_{\infty}(\cl A^{\star \star})$ (we refer the reader to \cite{Ble-03}).
Therefore we may assume that there exist $a_{i\,j}\in \cl A^{\star \star}, i,j \in \bb N$ such that 
 $$P(u)=(a_{i\,j})\cdot u\,,\forall\,\, \,u\,\in\, C_{\infty}(\cl A).$$ 
 In what follows we identify $P$ with the matrix $(a_{i\,j}).$ We also may consider a Hilbert space $K$ such that $\cl A\subseteq \cl A^{\star \star}\subseteq \mathbb{B}(K)$ and also $I_{K}\in\cl A^{\star \star}$.

Let $N_2$ be the linear span of the element $P.$ Since $P^2=P=P^{\star}$ we get that $N_2$ is a $\sigma$-TRO.
Let $\cl A^{1}=\overline{[\cl A+\mathbb{C}\,I_{K}]}$ and $N_1=C_{\infty}(\cl A^{1}).$ Clearly $N_1$ is a $\sigma$-TRO.
If $D$ is the $C^{\star}$-algebra generated by $P$  and $K_{\infty}(\cl A^{1})$, then $M=\overline{[N_2\,D\,N_1]}$ is a $\sigma$-TRO, \cite[Lemma 2.5]{Elest}.

 We note that $\overline{[M^{\star}\,M\,\cl A]}=\overline{[N_1^{\star}\,D\,N_2^{\star}\,N_2\,D\,N_1\,\cl A]}=\overline{[N_1^{\star}\,D\,N_2^{\star}\,C_{\infty}(\cl A^{1})\cl A]}\subseteq \cl A$.
If $Y_0=\overline{[M\,\cl A]},$  then  $$Y_0=\overline{[N_2\,D\,N_1\,\cl A]}=\overline{[N_2\,D\,C_{\infty}(\cl A)]}=\overline{[P\,D\,C_{\infty}(\cl A)]}=P(C_{\infty}(\cl A)).$$ We have that
\begin{align*}
    \overline{[M^{\star}\,M\,\cl A]}&=\overline{[M^{\star}\,P(C_{\infty}(\cl A))]}\\&=\overline{[N_1^{\star}\,D\,N_2^{\star}\,P\,C_{\infty}(\cl A)]}\\&=\overline{[N_1^{\star}\,D\,P\,C_{\infty}(\cl A)]}\\&=\overline{[R_{\infty}(\cl A^{1})\,P(C_{\infty}(\cl A))]}\\&=\overline{[R_{\infty}(\cl A)\,P(C_{\infty}(\cl A))]}
\end{align*}
and therefore $$(\overline{[M^{\star}\,M\,\cl A]})^{\star}=(\overline{[R_{\infty}(\cl A)\,P(C_{\infty}(\cl A))]})^{\star}\iff \overline{[\cl A\,M^{\star}\,M]}=\overline{[R_{\infty}(\cl A)\,P(C_{\infty}(\cl A))]}=\overline{[M^{\star}\,M\,\cl A]},$$ which implies that $\overline{[M^{\star}\,M\,\cl A]}=\overline{[\cl A\,M^{\star}\,M]}\subseteq \cl A.$ Since also $Y\cong P(C_{\infty}(\cl A))=Y_0=\overline{[M\,\cl A]}$, we conclude that $Y$ is a doubly $\sigma\Delta$ Hilbert module.\\

$iii) \implies ii)$ 

This is obvious.\\

$ii) \implies i)$ 

This is consequence of Theorem \ref{orthog}.

\end{proof}

At this point, we prove a Lemma which will be very useful for what follows.

\begin{lemma}
\label{cai}

Let $\cl A$ be an operator algebra with cai $(a_k)_{k\in K}$ and $\cl C$ be a $C^{\star}$-algebra with cai $(c_i)_{i\in I}$.
Assume that $\cl C\,\cl A\subseteq \cl A\,,\cl A\,\cl C\subseteq \cl A$.
We define $\cl A_{0}=\overline{[\cl C\,\cl A\,\cl C]}\subseteq \cl A$.
Then $\cl A_{0}$ is an operator algebra with a two-sided approximate identity $$x_{(i,k)}=c_i\,a_k\,c_i\,,i\in I\,,k\in K.$$

\end{lemma}

\begin{proof}

The space $\cl A_{0}$ is a closed subspace of $\cl A$ and is an algebra since $$\cl A_{0}\,\cl A_{0}\subseteq \overline{[\cl C\,\cl A\,\cl C\,\cl C\,\cl A\,\cl C]}\subseteq \overline{[C\,\cl A\,C\,\cl A\,\cl C]}\subseteq \overline{[C\,\cl A\,\,\cl A\,\cl C]}\subseteq \overline{[\cl C\,\cl A\,\cl C]}=\cl A_{0}.$$
It is obvious that $x_{(i,k)}=c_i\,a_k\,c_i\in \cl A_{0}\,,i\in I\,,k\in K$ and $\cl A_{0}\subseteq \cl A.$ 
Now, if $a\in \cl A_{0}$, then $c_i\,a\to a$ and $a_k\,a\to a$.
For all $i\in I\,,k\in K$ we have
\begin{align*}
    ||x_{(i,k)}\,a-a||&=||c_i\,a_k\,c_i\,a-a||\\&\leq ||c_i\,a_k\,c_i\,a-c_i\,a||+||c_i\,a-a||\\&\leq ||a_k\,c_i\,a-a||+||c_i\,a-a||\\&\leq ||a_k\,c_i\,a-a_k\,a||+||a_k\,a-a||+||c_i\,a-a||\\&\leq ||c_i\,a-a||+||a_k\,a-a||+||c_i\,a-a||\\&=2\,||c_i\,a-a||+||a_k\,a-a||
\end{align*}
Thus, $$\lim_{(i,k)}x_{(i,k)}\,a=a.$$ Similarly, we can prove that $$\lim_{(i,k)}a\,x_{(i,k)}=a.$$

\end{proof}

\begin{lemma}
\label{useful}

Let $\cl A\subseteq \mathbb{B}(H)$ be an approximately unital operator algebra and $M\subseteq \mathbb{B}(H,K)$ be a $\sigma$-TRO such that $M^{\star}\,M\,\cl A\subseteq \cl A$.
We also assume that $\cl A\,M^{\star}\,M\subseteq \cl A.$ We define $\cl B=\overline{[M\,\cl A\,M^{\star}]}\subseteq \mathbb{B}(K)$ and also $\cl A_{0}=\overline{[M^{\star}\,\cl B\,M]}\subseteq \mathbb{B}(H).$ Then $\cl A_{0}$ and $\cl B$ are approximately unital operator algebras and $\cl A_{0}\sim_{\sigma\,TRO}\cl B.$

\end{lemma}

\begin{proof}

It is sufficient to prove that $\cl A_{0}\,,\cl B$ are closed under multiplication and that $\cl A_{0}\sim_{\sigma\,TRO}\cl B.$ Indeed, $$\cl B\,\cl B\subseteq \overline{[M\,\cl A\,M^{\star}\,M\,\cl A\,M^{\star}]}\subseteq \overline{[M\,\cl A\,\cl A\,M^{\star}]}=\overline{[M\,\cl A\,M^{\star}]}=\cl B$$ so $\cl B$ is an operator algebra.
Now, we observe that $M\,M^{\star}\,\cl B\subseteq \cl B$ and then $$\cl A_{0}\,\cl A_{0}\subseteq \overline{[M^{\star}\,\cl B\,M\,M^{\star}\,\cl B\,M]}\subseteq \overline{[M^{\star}\,\cl B\,\cl B\,M]}\subseteq \overline{[M^{\star}\,\cl B\,M]}=\cl A_{0}$$ which means that $\cl A_{0}$ is an operator algebra.
We have that $\cl A_{0}=\overline{[M^{\star}\,\cl B\,M]}=\overline{[M^{\star}\,M\,\cl A\,M^{\star}\,M]}.$  If $C$ is the $C^*$-algebra $\overline{[M^{\star}\,M]}$, then  $C\,\cl A\subseteq \cl A\,\,,\cl A\,C\subseteq \cl A$.
By Lemma \ref{cai}, the operator algebra $\cl A_{0}$ has a cai.
Also, since $\cl A_{0}=\overline{[M^{\star}\,\cl B\,M]}$ and on the other hand $$\overline{[M\,\cl A_{0}\,M^{\star}]}=\overline{[M\,M^{\star}\,\cl B\,M\,M^{\star}]}=\overline{[M\,M^{\star}\,M\,\cl A\,M^{\star}\,M\,M^{\star}]}=\overline{[M\,\cl A\,M^{\star}]}=\cl B$$ we deduce that $\cl A_{0}\sim_{\sigma TRO}\cl B.$ Since $\cl A_{0}$ has a cai, we have that $\cl B$ has also a cai.

\end{proof}

\begin{theorem}
\label{TRO}

Let $\cl A$ be an approximately unital operator algebra and $Y$ be a doubly $\sigma \Delta$-$\cl A$-rigged module.
Then, there exist operator algebras $\cl A_{0}\,,\cl B$ with cai's such that $\cl A_{0}\sim_{\sigma TRO}\cl B$ and also $\cl B\sim_{\sigma TRO}Y.$ In case $\cl A$ is a $C^*-$algebra and $Y$ is a  $\sigma \Delta$-$\cl A$-Hilbert module then $\cl A_{0}\simeq I_{\cl A}(Y), B\simeq K_{\cl A}(Y).$

\end{theorem}

\begin{proof}

Let $H$ be a Hilbert space, $\,a:\cl A\to a(\cl A)\subseteq \mathbb{B}(H)$ be a completely isometric representation of $\cl A$ on $H$, and let $M\subseteq \mathbb{B}(H,K)$ be a $\sigma$-TRO such that $M^{\star}\,M\,a(\cl A)\subseteq a(\cl A)$ and also $$\overline{[M^{\star}\,M\,a(\cl A)]}=\overline{[a(\cl A)\,M^{\star}\,M]}\,\,(1)$$ Consider now a complete surjective isometry $$\phi:Y\to Y_0=\overline{[M\,a(\cl A)]}\subseteq \mathbb{B}(H,K)$$ which is a right $\cl A$-module map.
We define the spaces $\cl B=\overline{[M\,a(\cl A)\,M^{\star}]}\subseteq \mathbb{B}(K)$ and $\cl A_{0}=\overline{[M^{\star}\,\cl B\,M]}\subseteq \mathbb{B}(H)$. Now by Lemma \ref{useful}, $\cl A_{0}\,,\cl B$ are operator algebras with cai's such that $\cl A_{0}\sim_{\sigma\,TRO}\cl B.$  It remains to prove that $\cl B\sim_{\sigma TRO} Y$.
Set $M_1=M^{\star}\subseteq \mathbb{B}(K,H)$ and $M_2=\overline{[M\,M^{\star}]}\subseteq \mathbb{B}(K)$.
Then, $M_1\,,M_2$ are $\sigma$-TRO's and we have that $$\overline{[M_2^{\star}\,\phi(Y)\,M_1]}=\overline{[M\,M^{\star}\,M\,a(\cl A)\,M^{\star}]}=\overline{[M\,a(\cl A)\,M^{\star}]}=\cl B$$
and $$
    \overline{[M_2\,\cl B\,M_1^{\star}]}=\overline{[M\,M^{\star}\,M\,a(\cl A)\,M^{\star}\,M]}=\overline{[M\,a(\cl A)\,M^{\star}\,M]}\stackrel{(1)}{=}\overline{[M\,M^{\star}\,M\,a(\cl A)]}=\overline{[M\,a(\cl A)]}=\phi(Y).$$
Now by Definition \ref{eq}, $\cl B\sim_{\sigma TRO}Y.$

The other assertions follow easily.
 
\end{proof}

\begin{theorem}
\label{restr-doubly}

Let $\cl A$ be an approximately unital operator algebra, $a:\cl A\to \mathbb{B}(H)$ be a completely isometric homomorphism and $\,M\subseteq \mathbb{B}(H,K)$ be a $\sigma$-TRO such that $M^{\star}\,M\,a(\cl A)\subseteq a(\cl A).$ 
We define the $\sigma \Delta$-$\cl A$-rigged module $Y=\overline{[M\,a(\cl A)]}$.
Then there exist operator algebras $\cl A_{0}\,,\cl B$ with cai's and a restriction $Z$ of $\,Y$ such that  $\,Z$ is a doubly $\sigma \Delta$-$\cl A_{0}$-rigged module and $\cl A_{0}\sim_{\sigma TRO}\cl B\sim_{\sigma TRO}Z.$

\end{theorem}

\begin{proof}

 We define the restriction $Z=\overline{[Y\,M^{\star}\,M]}=\overline{[M\,a(\cl A)\,M^{\star}\,M]}$ of $\,Y.$ Let $\cl A_0=\overline{[\,M^{\star}\,M a(\cl A)\,M^{\star}\,M]}$ then 
$\cl A_0$ is an operator algebra and $$\overline{[M\,\cl A_{0}]}=\overline{[M\,M^{\star}\,M\,a(\cl A)\,M^{\star}\,M]}=\overline{[M\,a(\cl A)\,M^{\star}\,M]}=Z$$ 
such that $$\overline{[\,M^{\star}M\,\cl A_{0}]}=\overline{[M^{\star}\,M\,M^{\star}\,M\,a(\cl A)\,M^{\star}\,M]}=\overline{[M^{\star}\,M\,a(\cl A)\,M^{\star}\,M]}=\cl A_{0}$$
$$\overline{[\cl A_{0}\,M^{\star}\,M]}=\overline{[M^{\star}\,M\,a(\cl A)\,M^{\star}\,M\,M^{\star}\,M]}=\overline{[M^{\star}\,M\,a(\cl A)\,M^{\star}\,M]}=\cl A_{0}$$
which means that $\overline{[M^{\star}\,M\,\cl A_{0}]}=\overline{[\cl A_{0}\,M^{\star}\,M]}$, that is, $Z=\overline{[M\,\cl A_{0}]}$ is a doubly $\sigma \Delta$-$\cl A_{0}$-rigged module.
If we define $\cl B=\overline{[M\,a(\cl A)\,M^{\star}]}$ then by Lemma \ref{cai}, $\cl A_{0}$ and $\cl B$ have cai's and by Lemma \ref{useful} $\cl A_{0}\sim_{\sigma TRO}\cl B.$

Finally, $\cl B\sim_{\sigma TRO}Z$.
Indeed, if we consider the $\sigma$-TRO's $M_1=M$ and $M_2=\overline{[M\,M^{\star}]}$, then $$\overline{[M_2\,Z\,M_1^{\star}]}=\overline{[M\,M^{\star}\,M\,a(\cl A)\,M^{\star}\,M\,M^{\star}]}=\overline{[M\,\cl A_{0}\,M^{\star}]}=\cl B$$
$$\overline{[M_2^{\star}\,\cl B\,M_1]}=\overline{[M\,M^{\star}\,M\,a(\cl A)\,M^{\star}\,M]}=\overline{[M\,M^{\star}\,\cl B\,M]}=\overline{[M\,\cl A_{0}]}=Z.$$

\end{proof}

\begin{corollary}

Every $\sigma \Delta$-$\cl A$-rigged-module $Y$ over an approximately unital operator algebra $\cl A$  has a restriction which is a bimodule of BMP equivalence,
which actually implements a stable isomorphism over the operator algebras $\cl A_{0}$ and $\cl B$ defined as in Theorem \ref{restr-doubly}.

\end{corollary}

\begin{corollary}

Every orthogonally complemented module over an approximately unital operator algebra $\cl A$  has a restriction which is a bimodule of BMP equivalence between operator algebras which are stably isomorphic.

\end{corollary}

\begin{proof}

If $Y$ is an orthogonally complemented module over the operator algebra $\cl A$, then according to Theorem \ref{orthog}, $Y$ is a $\sigma \Delta$-$\cl A$-rigged module and due to the previous corollary, $Y$ has a restriction which is a bimodule of BMP equivalence between operator algebras which are stably isomorphic.

\end{proof}

Another interesting category of rigged modules is the category of column stable generator modules.
We prove that the restriction of a $\sigma \Delta$-rigged module over $\cl A$ is a column stable generated module (maybe over another operator algebra than $\cl A$).
We refer the reader to \cite[Section 8]{Ble-Gen} for facts about column stable generated modules.

\begin{definition}, \cite{Ble-Gen}.

A right $\cl A$-rigged module $Y$ is called a column stable generator (CSG for short) if there exist completely contractive right $\cl A$-module maps $\sigma:\cl A\to C_{\infty}(Y)$ and $\tau:C_{\infty}(Y)\to \cl A$ such that $\tau\circ \sigma=Id_{\cl A}.$

\end{definition}

\begin{proposition}

Let $\cl A$ be an approximately unital operator algebra, $a:\cl A\to a(\cl A)\subseteq \mathbb{B}(H)$ be a completely isometric homomorphism, and suppose there is a $\sigma$-TRO $M\subseteq \mathbb{B}(H,K)$ such that $$M^{\star}\,M\,a(\cl A)\subseteq a(\cl A), \,\,a(\cl A)\,M^{\star}\,M\subseteq a(\cl A).$$ Consider the $\sigma \Delta$-$\cl A$-rigged module $Y=\overline{[M\,a(\cl A)]}$.
Then, there exist operator algerbas $\cl A_{0}$ and $
\cl B$ and a restriction $Z$ of $\,Y$ over $\cl A_{0}$ such that $Z$ is a CSG module over $\cl A_{0}.$

\end{proposition}

\begin{proof}

Since $M$ is a $\sigma$-TRO, we fix a sequence $\left\{m_i\in M: i\in\mathbb{N}\right\}\subseteq M$ such that $$\nor{\sum_{i=1}^n m_i^{\star}\,m_i}\leq 1\,,\forall\,n\in\mathbb{N}\,\,\,,\sum_{i=1}^{\infty}m_i^{\star}\,m_i\,m^{\star}=m^{\star}\,,\forall\,m\in M\,\,(I).$$
We define the operator algebras $\cl B=\overline{[M\,a(\cl A)\,M^{\star}]}\subseteq \mathbb{B}(K)\,,\cl A_{0}=\overline{[M^{\star}\,\cl B\,M]}\subseteq \mathbb{B}(H)$ and also $Z=\overline{[Y\,M^{\star}\,M]}=\overline{[\cl B\,M]}$, which is a restriction of $Y$, and is also a doubly $\sigma \Delta$-$\cl A_{0}$-rigged module (Theorem \ref{restr-doubly} above).
Since $$\overline{[M\,\cl A_{0}]}= \overline{[M\,M^{\star}\,\cl B\,M]}=\overline{[M\,M^{\star}\,M\,a(\cl A)\,M^{\star}\,M]}=\overline{[M\,a(\cl A)\,M^{\star}\,M]}=\overline{[\cl B\,M]}=Z$$
and $\overline{[M^{\star}\,Z]}= \overline{[M^{\star}\,\cl B\,M]}=\cl A_{0}$, the maps 
$$\sigma:\cl A_{0}\to C_{\infty}(Z)\,,\sigma(a)=(m_i\,a)_{i\in\mathbb{N}}$$ and $$\tau:C_{\infty}(Z)\to \cl A_{0}\,,\tau((z_i)_{i\in\mathbb{N}})=\sum_{i=1}^{\infty}m_i^{\star}\,z_i$$ 
are well defined and also completely contractive right $\cl A_{0}$-module maps.
For all $m^{\star}\,b\,n\in M^{\star}\,\cl B\,M\subseteq \cl A_{0}$ we have that $$(\tau\circ \sigma)(m^{\star}\,b\,n)=\tau((m_i\,m^{\star}\,b\,n)_{i\in\mathbb{N}})=\sum_{i=1}^{\infty}m_i^{\star}\,m_i\,m^{\star}\,b\,n\stackrel{(I)}{=}m^{\star}\,b\,n=Id_{\cl A_{0}}(m^{\star}\,b\,n).$$
It follows that $(\tau\circ \sigma)(a)=Id_{\cl A_{0}}(a)\,,\forall\,a\in \cl A_{0}\implies \tau\circ \sigma=Id_{\cl A_{0}}.$

\end{proof}

\begin{theorem}
\label{main}

Let $\cl A\,,\cl B$ be approximately unital operator algebras such that $\cl A\,,\cl B$ are stably isomorphic.
Then, there exists a doubly $\sigma \Delta$-$\cl A$-rigged module $Y$ which is also a $\cl A-\cl B$-operator module and there exists a $\cl B$-$\cl A$-operator module $X$ such that $\cl B\cong Y\otimes_{\cl A}^h X$ and $\cl A\cong X\otimes_{\cl B}^h Y.$ Furthermore, $\cl A\,,\cl B\,,X\,,Y$ are all stably isomorphic.

\end{theorem}

\begin{proof}
Since $\cl A$ and $\cl B$ are stably isomorphic, we have that they are also $\sigma \Delta$ equivalent, that is, $\cl A\sim_{\sigma \Delta}\cl B$ (according to \cite[Theorem 3.3]{Elest}).
So, there exist Hilbert spaces $H\,,K$ and completely isometric homomorphisms $a:\cl A\to \mathbb{B}(H)$ and $\beta:\cl B\to \mathbb{B}(K)$ and also a $\sigma$-TRO $M\subseteq \mathbb{B}(H,K)$ such that $$a(\cl A)=\overline{[M^{\star}\,\beta(\cl B)\,M]}\,\,,\beta(\cl B)=\overline{[M\,a(\cl A)\,M^{\star}]}.$$
We have that $$\overline{[a(\cl A)\,M^{\star}\,M]}=a(\cl A)=\overline{[M^{\star}\,M\,a(\cl A)]},$$
and so $Y=\overline{[M\,a(\cl A)]}\subseteq \mathbb{B}(H,K)$ is a doubly $\sigma \Delta$-$\cl A$-rigged module which is also a left $\cl B$-operator module since $$\beta(\cl B)\,Y\subseteq \overline{[M\,a(\cl A)\,M^{\star}\,M\,a(\cl A)]}\subseteq \overline{[M\,a(\cl A)\,a(\cl A)]}\subseteq \overline{[M\,a(\cl A)]}=Y.$$ 
We also define $X=\overline{[a(\cl A)\,M^{\star}]}\subseteq \mathbb{B}(K,H)$ which is a left $\cl A$-operator module via the module action $$a(x)\cdot (a(y)\,m^{\star})=a(x\,y)\,m^{\star}\,,x\,,y\in\cl A\,,m\in M.$$ Furthermore $X$ is a right $\cl B$-operator module since $$X\,\beta(\cl B)\subseteq \overline{[a(\cl A)\,M^{\star}\,M\,a(\cl A)\,M^{\star}]}=\overline{[a(\cl A)\,a(\cl A)\,M^{\star}\,M\,M^{\star}]}\subseteq \overline{[a(\cl A)\,M^{\star}]}=X.$$
By Lemma \ref{a lot}, if $D_1=\overline{[M^{\star}\,M]}$, then
    \begin{align*} Y\otimes_{a(\cl A)}^h X&=\overline{[M\,a(\cl A)]}\otimes_{a(\cl A)}^h \overline{[a(\cl A)\,M^{\star}]}\\&\cong \left(M\otimes_{D_1}^h a(\cl A)\right)\otimes_{a(\cl A)}^h \left(a(\cl A)\otimes_{D_1}^h M^{\star}\right)\\&\cong M\otimes_{D_1}^h a(\cl A)\otimes_{D_1}^h M^{\star}\\&\stackrel{(1.1)}{\cong}\overline{[M\,a(\cl A)\,M^{\star}]}=\beta(\cl B)
\end{align*}
and also, due to the fact that $Y=\overline{[a(\cl A)\,M^{\star}]}=\overline{[M^{\star}\,\beta(\cl B)]}$ if $D_2=\overline{[M\,M^{\star}]}$, we have
\begin{align*}
    X\otimes_{\beta(\cl B)}^h Y&=\overline{[M^{\star}\,\beta(\cl B)]}\otimes_{\beta(\cl B)}^h \overline{[M\,M^{\star}\,\beta(\cl B)\,M]}\\&\cong \left(M^{\star}\otimes_{D_2}^h \beta(\cl B)\right)\otimes_{\beta(\cl B)}^h \overline{[M\,a(\cl A)]}\\&\cong M^{\star}\otimes_{D_2}^h \left(\beta(\cl B)\otimes_{\beta(\cl B)}^h \overline{[M\,a(\cl A)]}\right)\\&\cong M^{\star}\otimes_{D_2}^h \overline{[M\,a(\cl A)]}\\&\cong M^{\star}\otimes_{D_2}^h\left(M\otimes_{D_1}^h a(\cl A)\right)\\&\cong \left(M^{\star}\otimes_{D_2}^h M\right)\otimes_{D_1}^h a(\cl A)\\&\cong \overline{[M^{\star}\,M]}\otimes_{D_1}^h a(\cl A)\\&\cong \overline{[M^{\star}\,M\,a(\cl A)]}\\&=\overline{[M^{\star}\,M\,M^{\star}\,\beta(\cl B)\,M]}\\&=\overline{[M^{\star}\,\beta(\cl B)\,M]}=a(\cl A).
\end{align*}

\end{proof}

By the same arguments we obtain the following corollary:

\begin{corollary} Let $\cl A, \cl B$  be stably isomorphic $C^*-$algebras. There exists a $\sigma\Delta-$Hilbert module $Y$ over a $C^*-$algebra $\cl D$ such that $$\cl A\simeq K_{\cl D}(Y), \cl B\simeq I_{\cl D}(Y).$$ Furthermore $\cl A, \cl B$ and $Y$ are all stably isomorphic.

\end{corollary}

\section{Morita equivalence of rigged modules}\label{444}

\begin{definition}, \cite{Ble-Gen}.
Let $\cl A$ be an approximately unital operator algebra and let $Y$ be a right $\cl A$-rigged module.
If $\left\{\Phi_{b}\,,\Psi_{b}: b\in B\right\}$ is a choice for $Y$ as in Definition \ref{rigged}, then we write $E_{b}$ for the map $E_{b}=\Psi_{b}\circ \Phi_{b}:Y\to Y\,,b\in B$.
We define $$\tilde{Y}=\left\{f\in CB_{\cl A}(Y,\cl A): f\circ E_{b}\to f\,\,\text{uniformly}\right\}$$ and $\mathbb{K}(Y)$ to be the closure in $CB_{\cl A}(Y,Y)$ of the set of finite rank operators $$T_{y,f}:Y\to Y\,,T_{y,f}(y')=y\,f(y')$$ where $y\in Y\,,f\in\tilde{Y}$.

\end{definition}

For further details we refer the reader to \cite{Ble-Gen}, Section 3.
We note that $\mathbb{K}(Y)$ and $\tilde{Y}$ are actually independent of the particular directed set and nets $\left\{\Phi_{b}\,,\Psi_{b}: b\in B\right\}$.
In the following lemma we use the notion of a complete quotient map.
For further details we refer the reader to \cite{BMP00}.

\begin{lemma}
\label{after}

Let $\cl A\subseteq \mathbb{B}(H)$ be an approximately unital operator algebra, $\,M\subseteq \mathbb{B}(H,K)$ be a $\sigma$-TRO and $Y=\overline{[M\,\cl A]}\subseteq \mathbb{B}(H,K)$.
Assume that $M^{\star}\,M\,\cl A\subseteq \cl A\,\,,\cl A\,M^{\star}\,M\subseteq \cl A$
(thus $Y$ is a $\sigma \Delta$-$\cl A$-rigged module). Then $\tilde{Y}\cong \overline{[\cl A\,M^{\star}]}$ and $\mathbb{K}(Y)\cong \overline{[M\,\cl A\,M^{\star}]}$.

\end{lemma}

\begin{proof}

We define $\cl B=\overline{[M\,\cl A\,M^{\star}]}$.
Clearly, $\cl B$ is an operator algebra.
By Lemma \ref{cai}, $\cl A_{0}=\overline{[M^{\star}\,M\,\cl A\,M^{\star}\,M]}$ has a cai.
By Lemma \ref{useful} the algebra  $\cl B$ has also cai and obviously the algebras $\cl A_{0}$ and $ \cl B$ are $\sigma-TRO-$equivalent.
If $X=\overline{[\cl A\,M^{\star}]}$, then we define the completely contractive maps $$\left(\cdot\,,\cdot\right):X\times Y\to \cl A\,,(x,y)\mapsto (x,y)=x\,y$$ $$\left[\cdot\,,\cdot\right]:Y\times X\to \cl B\,,(y,x)\mapsto [y,x]=y\,x.$$
These maps satisfy $$(x,y)\,x'=x\,[y,x']\,\,,y\,(x,y')=[y,x]\,y'\,,\forall\,x\,,x'\in X\,,y\,,y'\in Y.$$
The map $\left[\cdot\,,\cdot\right]$ induces a complete quotient map $Y\otimes^h X\to \cl B\,,y\otimes x\to y\,x$.
Indeed, by making the same calculations as those of the proof of Theorem \ref{main}, we have that $Y\otimes_{\cl A}^h X\cong {[M\,\cl A\,M^{\star}]}=\cl B$.
Futhermore, the map $\phi: Y\otimes^h X\to Y\otimes_{\cl A}^h X\,,y\otimes x\mapsto y\otimes_{\cl A}x$ is a complete quotient since the map $\hat{\phi}:\left(Y\otimes^h X\right)/\rm{Ker}(\phi)\to Y\otimes_{\cl A}^h X$ is a complete surjective isometry.
By \cite[Theorem 5.1]{Ble-Gen}, $\tilde{Y}\cong \overline{[\cl A\,M^{\star}]}$ and $\mathbb{K}(Y)\cong \cl B=\overline{[M\,\cl A\,M^{\star}]}$.

\end{proof}

\begin{theorem}
If $\cl A$ is an approximately unital operator algebra and $Y$ is a doubly $\sigma \Delta$-$\cl A$-rigged module, then there exist approximately unital operator algebras $\cl A_{0}\subseteq \cl A$ and $\cl B$ such that:\\
$i)$ $\,\cl B\cong Y\otimes_{\cl A_{0}}^h \tilde{Y}$;\\
$ii)$ $\,\cl A_{0}\cong \tilde{Y}\otimes_{\cl B}^h Y$;\\
$iii)$ $\cl A_{0}\sim_{\sigma \Delta}\cl B\,\,,\cl A_{0}\sim_{\sigma \Delta}Y\,\,,Y\sim_{\sigma \Delta} \tilde{Y}$.
\end{theorem}

\begin{proof}
It suffices to prove the above assertions for the case of a doubly $\sigma$-TRO-$\cl A$-module $Y=\overline{[M\,\cl A]}$ where $\cl A\subseteq \mathbb{B}(H)\,,M\subseteq \mathbb{B}(H,K)$ is a $\sigma$-TRO such that $M^{\star}\,M\,\cl A\subseteq \cl A$ and $$\overline{[M^{\star}\,M\,\cl A]}=\overline{[\cl A\,M^{\star}\,M]}\,\,(1).$$ We set $\cl A_{0}=\overline{[\cl A\,M^{\star}\,M]}\subseteq \cl A$.
Clearly $\cl A_{0}$ is an approximately unital  operator algebra.
\\i) By Lemma \ref{after}, $\tilde{Y}\cong \overline{[\cl A\,M^{\star}]}$, and so
$$\overline{[\cl A_{0}\,M^{\star}]}=\overline{[\cl A\,M^{\star}\,M\,M^{\star}]}=\overline{[\cl A\,M^{\star}]}=\tilde{Y}$$
and on the other hand 
$$\overline{[M\,\cl A_{0}]}=\overline{[M\,\cl A\,M^{\star}\,M]}\stackrel{(1)}{=}\overline{[M\,M^{\star}\,M\,\cl A]}=\overline{[M\,\cl A]}=Y.$$  Using Lemma \ref{a lot} and making the same calculations as in the proof of Theorem \ref{main} we have that $Y\otimes_{\cl A_{0}}^h \tilde{Y}\cong \overline{[M\,\cl A\,M^{\star}]}$.
If we define $\cl B=\overline{[M\,\cl A\,M^{\star}]}$, then $\cl B$ is an approximately unital operator algebra such that $\cl B\cong Y\otimes_{\cl A_{0}}^h \tilde{Y}$.\\$ii)$ It is true that $\tilde{Y}\cong \overline{[\cl A\,M^{\star}]}=\overline{[M^{\star}\,\cl B]}$, so if $D_1=\overline{[M^{\star}\,M]}$ and $D_2=\overline{[M\,M^{\star}]}$, it follows that 
\begin{align*}
    \tilde{Y}\otimes_{\cl B}^h Y&=\overline{[M^{\star}\,\cl B]}\otimes_{\cl B}^h \overline{[M\,\cl A]}\\&\cong \left(M^{\star}\otimes_{D_2}^h \cl B\right)\otimes_{\cl B}^h Y\\&\cong M^{\star}\otimes_{D_2}^h \left[\cl B\otimes_{\cl B}^h \left(M\otimes_{D_1}^h \cl A\right)\right]\\&\cong M^{\star}\otimes_{D_2}^h M\otimes_{D_1}^h \cl A\\&\cong \overline{[M^{\star}\,M]}\otimes_{D_1}^h \cl A\\&\cong \overline{[M^{\star}\,M\,\cl A]}\\&=\overline{[\cl A\,M^{\star}\,M]}=\cl A_{0}.
\end{align*}\\
$iii)$ Consider the $\sigma$-TROs $M_1=M^{\star}\subseteq \mathbb{B}(K,H)$ and $M_2=M\subseteq \mathbb{B}(H,K)$. Then $$\overline{[M_2^{\star}\,Y\,M_1]}=\overline{[M^{\star}\,M\,\cl A\,M^{\star}]}=\overline{[\cl A\,M^{\star}\,M\,M^{\star}]}=\overline{[\cl A\,M^{\star}]}=\tilde{Y}$$ and $$\overline{[M_2\,\tilde{Y}\,M_1^{\star}]}=\overline{[M\,\cl A\,M^{\star}\,M]}=\overline{[M\,M^{\star}\,M\,\cl A]}=\overline{[M\,\cl A]}=Y$$
so $Y\sim_{\sigma TRO}\tilde{Y}$.
By Theorem \ref{TRO}, we also have that $\cl B\sim_{\sigma TRO}Y$ and $\cl B\sim_{\sigma \Delta} \cl A_{0}$.

\end{proof}

\begin{definition}
\label{def M}

Let $\cl A\,,\cl B$ be approximately unital operator algebras, $\,E$ be a right $\cl B$-rigged module and $F$ be a right $\cl A$-rigged module.
We call $E$ and $F$ Morita equivalent if there exists a right $\cl A$-rigged module $Y$ such that $\cl A\cong \tilde{Y}\otimes_{\cl B}^h Y\,\,,\cl B\cong Y\otimes_{\cl A}^h \tilde{Y}$ and also $F\cong E\otimes_{\cl B}^h Y.$ In this case we write $E\sim_{M} F$.

\end{definition}

\begin{remark}
\label{compact}

If $\cl A\,,\cl B\,,E$ and $F$ are as above (Definition \ref{def M}), then by \cite[Theorem 6.1]{Ble-Gen}, $$\mathbb{K}(F)\cong \mathbb{K}\left(E\otimes_{\cl B}^h Y\right)\cong \mathbb{K}(E).$$

\end{remark}

\begin{definition}
\label{def sM}

Let $\cl A\,,\cl B$ be approximately unital operator algebras, $\,E$ be a right $\cl B$-rigged module and $F$ be a right $\cl A$-rigged module.
We call $E$ and $F$ $\sigma$-Morita equivalent if there exists a doubly $\sigma \Delta$-$\cl A$-rigged module $Y$ such that $\cl A\cong \tilde{Y}\otimes_{\cl B}^h Y\,,\cl B\cong Y\otimes_{\cl A}^h \tilde{Y}$ and also $F\cong E\otimes_{\cl B}^h Y.$ In this case we write $E\sim_{\sigma M} F.$

\end{definition}

\begin{remark} Other notions of Morita equivalence for the subcategory of Hilbert modules exist in \cite{JM, Skeide}.
\end{remark}

\begin{proposition}

If $E\sim_{\sigma M}F$, then $\mathbb{K}(E)\cong \mathbb{K}(F).$

\end{proposition}

\begin{proof}
It is obvious that if $E\sim_{\sigma M} F$, then $E\sim_{M}F$ and thus  that if $E\sim_{\sigma M}F$, then, by Remark \ref{compact}, $\mathbb{K}(E)\cong \mathbb{K}(F).$
\end{proof}

\begin{lemma}
\label{useful 2}

Let $M$ be a $\sigma$-TRO, $D_1=\overline{[M\,M^{\star}]}\,\,,D_2=\overline{[M^{\star}\,M]}$, $\,E$ be a right $D_1$-module and $F$ be a right $D_2$-module such that $F\cong E\otimes_{D_1}^h M$.
Then $E\sim_{\sigma \Delta} F.$

\end{lemma}

\begin{proof}

By \cite[Theorem 3.8]{Ele-Pap}, it suffices to prove that $E$ and $F$ are stably isomorphic.
We may assume that $F=E\otimes_{D_1}^h M.$ Hence, 
\begin{align*}
    F\otimes_{D_2}^h M^{\star}&=\left(E\otimes_{D_1}^h M\right)\otimes_{D_2}^h M^{\star}\\&\cong E\otimes_{D_1}^h \left(M\otimes_{D_2}^h M^{\star}\right)\\&\cong E\otimes_{D_1}^h D_1\\&\cong E.
\end{align*}
Thus, we can also assume that there exists a complete onto isometry 
$$a:F\otimes_{D_2}^h M^{\star}\to E$$ such that $a((e\otimes_{D_1}m)\otimes_{D_2}n^{\star})=e\,m\,n^{\star}\,,\forall\,e\in E\,,m\,,n\in M\,(\ast).$ There exists a sequence $\left\{m_i\in M: i\in\mathbb{N}\right\}$ such that $$\nor{\sum_{i=1}^n m_i^{\star}\,m_i}\leq 1\,,\forall\,n\in\mathbb{N}$$ and also $$\sum_{i=1}^{\infty}m\,m_i^{\star}\,m_i=m\,,\forall\,m\in M.$$ We observe that for all $e\in E$ and $m\in M$ we have that $$ \sum_{i=1}^{\infty}a((e\otimes_{D_1}m)\otimes_{D_2}m_i^{\star})\otimes_{D_1}m_i\stackrel{(\ast)}{=}\sum_{i=1}^{\infty}e\,m\,m_i^{\star}\otimes_{D_1}m_i \stackrel{m\,m_i^{\star}\in D_1}{=}\sum_{i=1}^{\infty}e\otimes_{D_1}m\,m_i^{\star}\,m_i=e\otimes_{D_1}m.$$
Thus, $$\sum_{i=1}^{\infty}a(f\otimes_{D_2}m_i^{\star})\otimes_{D_1}m_i=f\,,\forall\,f\in F\,\,(\ast \ast).$$
We define the completely contractive maps $$\Phi:F\to R_{\infty}(E)\,,\Phi(f)=(a(f\otimes_{D_2}m_i^{\star}))_{i\in\mathbb{N}}$$ $$\Psi:R_{\infty}(E)\to F\,,\Psi((e_i)_{i\in\mathbb{N}})=\sum_{i=1}^{\infty}e_i\otimes_{D_1}m_i.$$
Using $(\ast \ast)$, we have that $$(\Psi\circ \Phi)(f)=\sum_{i=1}^{\infty}a(f\otimes_{D_2}m_i^{\star})\otimes_{D_1}m_i=f\,,\forall\,f\in F.$$
So, $\Phi$ is a complete isometry and $P=\Phi\circ \Psi: R_{\infty}(E)\to R_{\infty}(E)$ is a projection and we have that $\Phi(F)=\rm{Ran}(P).$ Now we employ the usual arguments, see for example the proof of Corollary 8.2.6 of \cite{BleLeM04}: 
$$R_{\infty}(E)\cong \rm{Ran}(P)\oplus_{r}\rm{Ran}(I-P)\cong \Phi(F)\oplus_{r}\rm{Ran}(I-P)\cong F\oplus_{r}\rm{Ran}(I-P)$$ where $I=I_{R_{\infty}(E)}.$ Thus, \begin{align*}
    R_{\infty}(E)&\cong R_{\infty}(R_{\infty}(E))\\&\cong (F\oplus_{r}\rm{Ran}(I-P))\oplus_{r}(F\oplus_{r}\rm{Ran}(I-P))\oplus_{r}...\\&\cong F\oplus_{r}(\rm{Ran}(I-P)\oplus_{r}F)\oplus_{r}(\rm{Ran}(I-P)\oplus_{r}F)\oplus_{r}...\\&\cong F\oplus_{r}R_{\infty}(E).
\end{align*}
Therefore, $R_{\infty}(E)\cong R_{\infty}(R_{\infty}(E))\cong R_{\infty}(F)\oplus_{r}R_{\infty}(E)$.
By symmetry, $R_{\infty}(F)\cong R_{\infty}(E)\oplus_{r}R_{\infty}(F)$, so $R_{\infty}(E)\cong R_{\infty}(F)$ which implies that $K_{\infty}(E)\cong K_{\infty}(F).$

\end{proof}

\begin{theorem}

Let $\cl A\,,\cl B$ be approximately unital operator algebras, $E$ be a right $\cl B$-rigged module and $F$ be a right $\cl A$-rigged module such that $E\sim_{\sigma M} F$.
Then $E\sim_{\sigma \Delta} F.$

\end{theorem}

\begin{proof}

Let $a:\cl A\to \mathbb{B}(H)$ be a completely-isometric representation of $\cl A$ on $H$ and $M\subseteq \mathbb{B}(H,K)$ be a $\sigma$-TRO such that $M^{\star}\,M\,a(\cl A)\subseteq a(\cl A)$ and also $\overline{[M^{\star}\,M\,a(\cl A)]}=\overline{[a(\cl A)\,M^{\star}\,M]}.$ Consider also the doubly $\sigma \Delta$-$\cl A$-rigged module $Y=\overline{[M\,a(\cl A)]}$ such that $a(\cl A)\cong \tilde{Y}\otimes_{\cl B}^h Y\,\,,\cl B\cong Y\otimes_{\cl A}^h \tilde{Y}\cong \overline{[M\,a(\cl A)\,M^{\star}]}$ and also $F\cong E\otimes_{\cl B}^h Y.$
We define $D_1=\overline{[M\,M^{\star}]}$ and we have that $\cl B\,M\,M^{\star}\subseteq \cl B$.
So $$E=\overline{[E\,\cl B]}\supseteq \overline{[E\,\cl B\,M\,M^{\star}]}=\overline{[E\,M\,M^{\star}]}$$ which means that $E$ is a right $D_1$-module.
Therefore, since $Y=\overline{[M\,a(\cl A)]}=\overline{[\cl B\,M]}$, it holds that 
$$F\cong E\otimes_{\cl B}^h Y=E\otimes_{\cl B}^h \overline{[\cl B\,M]}\cong E\otimes_{\cl B}^h \left(\cl B\otimes_{D_1}^h M\right)\cong \left(E\otimes_{\cl B}^h \cl B\right)\otimes_{D_1}^h M\cong E\otimes_{D_1}^h M.$$

Observe that if $D_2=\overline{[M^{\star}\,M]}$, then $F=\overline{[F\,\cl A]}\supseteq \overline{[F\,\cl A\,M^{\star}\,M]}=\overline{[F\,M^{\star}\,M]}$ which means that $F$ is a right $D_2$-module. Now by Lemma \ref{useful 2}, $E\sim_{\sigma \Delta}F.$

\end{proof}

\end{document}